\documentclass[]{jgcc}
\pdfoutput=1
\usepackage{lastpage}
\jgccdoi{18}{2}{4}{18259}
\jgccheading{}{\pageref{LastPage}}{}{}{May~22,~2026}{Oct.~2,~2026}{}

\usepackage{amssymb,bm,mathrsfs}
\usepackage{graphicx,rotating}
\usepackage{orcidlink}
\usetikzlibrary{arrows.meta,positioning}
\numberwithin{equation}{section}


\newcommand{\Real}{\mathbb R}
\newcommand{\Z}{\mathbb Z}

\newcommand{\aut}{\mathrm{Aut}}

\newcommand{\edo}{\mathrm{End}}
\newcommand{\mon}{\mathrm{Mon}}

\makeatletter
\newcommand*\bigcdot{\mathpalette\bigcdot@{1.5}}
\newcommand*\bigcdot@[2]{\mathbin{\vcenter{\hbox{\scalebox{#2}{$\m@th#1\bullet$}}}}}
\makeatother

\begin{document}

\title[Test elements for monomorphisms of free groups]{A note on test elements for monomorphisms of free groups}

\author[Q.~Zhang]{Qiang Zhang\texorpdfstring{\,\orcidlink{0000-0001-6332-5476}}{}}
\address{School of Mathematics and Statistics, Xi'an Jiaotong University, Xi'an 710049, CHINA}
\email{zhangq.math@mail.xjtu.edu.cn}
% Classification codes and edition preserved from the original manuscript.

\author[D.~Zhao]{Dongxiao Zhao\texorpdfstring{\,\orcidlink{0009-0006-7919-9244}}{}}
	\address{School of Mathematics and Statistics, Xi'an Jiaotong University, Xi'an 710049, CHINA}
	\email{zdxmath@stu.xjtu.edu.cn}

\thanks{The authors are partially supported by NSFC (No. 12471066 ) and the Shaanxi Fundamental Science Research Project for Mathematics and Physics (No. 23JSY027).}

\begin{abstract}
A word in a group is called a test element if any endomorphism fixing it is necessarily an automorphism.  In this note, we give a sufficient condition in geometry to construct test elements for monomorphisms of a free group, by using the Whitehead graph and the action of the free group on its Cayley graph.
\end{abstract}
\maketitle

\section{Introduction}

For a group $G$, let $\edo(G)$ (resp. $\mon(G), \aut(G)$) denote the set of all the endomorphisms (resp. monomorphisms, automorphisms) of $G$.
The concept of test elements was introduced by Shpilrain \cite{Sh94}: an element $w\in G$ is called a \textit{test element} in $G$ if  for any $\phi \in \edo(G)$, $\phi(w)=w$ implies $\phi\in\aut(G).$ Moreover,  $w$ is called a \textit{test element for monomorphisms}, if for any monomorphism $\phi\in\mon(G)$, $\phi(w)=w$ implies $\phi \in \aut(G)$.  %Another equivalent definition of test elements was given in \cite{OT00}: $w$ is a test element if for any $\phi \in \edo(G)$, there is some $\alpha\in\aut(G)$ such that $\phi(w)=\alpha(w)$ implies $\phi \in \aut(G)$.  
Note that if $w\in G$ is a test element (for monomorphisms), then for every $\phi \in \aut(G)$, $\phi(w)$ is again a test element  (for monomorphisms).

For test elements in free groups,  we have:

\begin{exa}\label{commutator is test element}
 Suppose $F_{r}$ is a free group of rank $r$ with basis $\{x_1,x_2,\ldots,x_{r}\}$. Then the element  $a_1^ka_2^k\cdots a_r^k (k\geq 2)$ is a test element.  Moreover, if $r=2n$ is even, then the commutator
 $$[x_1,x_2][x_3,x_4]\cdots[x_{2n-1},x_{2n}]$$ is a test element. See \cite{Zie66, Sh94, Tu96} etc. 
\end{exa}

Turner  \cite{Tu96} studied the test elements in a free group $F_r$ and proved that the free group satisfies the \textit{Retract Theorem}: $w \in F_r$ is a test element if and only if $w$ is not contained in any proper retract of $F_r$. So a group satisfying the Retract Theorem is called a \textit{Turner group}.  Note that the Retract Theorem also holds for any monounary algebra \cite{JP07}. 

There are many studies on Turner groups in the literature: O'Neill and Turner \cite{OT00,OT01} studied the test elements in direct products and showed that torsion-free stably hyperbolic groups and finitely generated Fuchsian groups are both Turner groups.
Moreover, Groves \cite{Gro12} showed the assumption ``stably" can be removed:  all torsion-free hyperbolic groups are Turner groups. Snopce and Tanushevski studied test elements in a group by considering its pro-$p$ completion and showed that finitely generated profinite groups are Turner groups; furthermore, they showed the set of test elements of a free group and a surface group of genus $n \geq 2$ (orientable) or $n \geq 3$ (non-orientable) is dense in the profinite topology, see \cite{ST171, ST172, ST173}.
Note that O'Neill \cite{ON19} showed the solvable Baumslag-Solitar group $BS_{1,n}=\langle a,t | tat^{-1}=a^n \rangle $ is not a Turner group. 

In this note, we study the test elements for monomorphisms in a free group $F_r$. Turner \cite{Tu96} showed that $a \in F_r$ is a test element for monomorphisms if and only if $a$ is not contained in any proper free factor of $F_r$. For a given element $a \in F_r$ and a test element $w \in F_r$, we will give a sufficient condition in geometry to determine whether $a$ is a test element for monomorphisms, by using the Whitehead graph and the action of $F_r$ on its Cayley graph, see Section \ref{sect 2} and Section \ref{sect 3} for more details.

\section{Simple words and Whitehead graphs}\label{sect 2}

Let $F_r$ be a free group of rank $r$ and let $\mathcal{S}$ be a basis of  $F_r$.

\subsection{Reduced words, free factors and simple words}

Every element $u\in F_r$ is presented by a unique reduced word in the letters of $\mathcal{S} \cup \mathcal{S}^{-1}$, the length of this reduced word is denoted by $|u|_\mathcal{S}$. We write $=$ to denote equality in $F_r$ and $\equiv$ to denote equality as words.  A word $u_1$ is a \textit{prefix} of $u$, if there is a word $u_2$ such that $u\equiv u_1 u_2$.
Every element $w \in F_r$ has a \textit{cyclic reduction} $w_0$, obtained by writing
$w \equiv w_1w_0w_1^{-1}$ with $w_1$ as long as possible. An element is \textit{cyclically reduced} (with respect to the basis $\mathcal{S}$ ) if $w \equiv w_0$.  

A subgroup $H\leq F_r$ is a \textit{free factor} if it is generated by a subset of a basis, and is a \textit{proper free factor} if, additionally, $H\neq F_r$. An element in $F_r$ is called \textit{primitive} if it belongs to some basis of $F_r$, and is called \textit{simple} if it belongs to some proper free factor of $F_r$.  Note that a primitive element is always simple, and a power of a simple element is again simple. 

 Note that $|u|_\mathcal{S}=|\phi(u)|_\mathcal{\phi(S)}$ for every automorphism $\phi \in \aut(F_r)$. We have:
 
\begin{lem}\label{def of minimizes length}
Let $u \in F_r$. A basis $\mathcal{S}$ minimizes $|u|_\mathcal{S}$, i.e., $|u|_{\mathcal{S}}\leq |u|_{\mathcal{S}'}$ for any basis $\mathcal{S}'$ of $F_r$, if and only if  $|u|_\mathcal{S} \leq |\phi(u)|_\mathcal{S}$ for every automorphism $\phi \in \aut(F_r)$.
\end{lem}

\begin{lem}\label{min length lem}
Let $u \in F_r$. If the basis $\mathcal{S}$ minimizes $|u|_\mathcal{S}$, then $u$ is cyclically reduced with respect to $\mathcal{S}$.
\end{lem}

\begin{proof}
	Suppose $u$ is not cyclically reduced, then there is a presentation
 $$u \equiv s_1 w s_1^{-1},$$
	where $s_1 \in \mathcal{S}$. We can find another basis of $F_r$, denoted by $\mathcal{S}'$, whose elements are obtained by the conjugation with $s_1$ of the elements in $\mathcal{S}$. It is clear that $\mathcal{S}'$ is a basis of $F_r$, and $|u|_{\mathcal{S}'}<|u|_\mathcal{S}$ which contradicts our assumption.
\end{proof}

\begin{prop}\label{commutator is minimal}
Suppose $F_{r}$ is a free group of rank $r$ with basis $\mathcal{S}=\{x_1,x_2,\ldots,x_{r}\}$. 
\begin{enumerate}
    \item Let $u=x_1^2 x_2^2\cdots x_r^2$. Then $\mathcal{S}$ minimizes $|u|_\mathcal{S}$.

\item If $r=2n$ is even, and $w=[x_1,x_2][x_3,x_4]\cdots[x_{2n-1},x_{2n}],$
 then $\mathcal{S}$ minimizes $|w|_\mathcal{S}$.
\end{enumerate}
\end{prop}

\begin{proof}
By Lemma \ref{def of minimizes length}, to prove $\mathcal{S}$ minimizes $|u|_\mathcal{S}$, it  is equivalent to prove
$|u|_\mathcal{S} \leq |\phi(u)|_\mathcal{S}$ for every automorphism $\phi \in \aut(F_r)$.  

(1) Let $[F_r,F_r]$ be the commutator subgroup of $F_r$ and let
$$\widetilde F_r:=\{x^2y|x \in F_r, y \in [F_r,F_r]\}.$$
We claim that $\widetilde F_r$ is a characteristic subgroup of $F_r$.  Indeed,  any finite product of square elements can be presented as a product of a square element and an element in $[F_r,F_r]$, i.e., for any $a, b\in F_r$, the product
$$a^2b^2=(ab)^{2}\cdot b^{-1}\cdot a^{-1}b^{-1}ab\cdot b\in (ab)^2[F_r, F_r]\subset \widetilde F_r.$$
Moreover, for any $y, t\in[F_r,F_r]$, we have $b^2yt^{-1}b^{-2}\in [F_r,F_r]$ and hence
$$(a^2y)(b^2t)^{-1}=a^2b^{-2}\cdot b^2yt^{-1}b^{-2}\in \widetilde F_r,$$
which implies that $\widetilde F_r$ is a subgroup of $F_r$. It is clear $\phi(\widetilde F_r)=\widetilde F_r$ for any $\phi\in\aut(F_r)$,  and hence our claim holds.

Now we consider the abelianization  $\pi: F_r\to \Z^r$. For any $\phi \in \aut(F_r)$, note that $u=x_1^2 x_2^2\cdots x_r^2\in \widetilde F_r$ and $\pi(\widetilde F_r) \subset (2\Z)^r$, we have $\phi(u) \in \widetilde F_r$ and
\begin{equation}\label{eq. phi(w) even}
  \pi(\phi(u)) \in (2\Z)^r.
\end{equation}
Moreover, Example \ref{commutator is test element} shows that $u$ is a test element in $F_r$, thus $\phi(u) \in \widetilde F_r$ is also a test element in $F_r$, and cannot belong to any proper free factor of $F_r$ by Proposition \ref{Turner2}. So for any $1 \leq i \leq r$, $x_i$ or $x_i^{-1}$ occurs at least once in the reduced presentation of $\phi(u)$ with respect to $\mathcal{S}$. Then combining Equation (\ref{eq. phi(w) even}), we have 
$|\phi(u)|_\mathcal{S}\geq 2r=|u|_\mathcal{S}$. \\

(2) As in the above case, we consider the abelianization  $\pi: F_{2n}\to \Z^{2n}$. For any $\phi \in \aut(F_{2n})$, note that 
$$w =[x_1,x_2][x_3,x_4]\cdots[x_{2n-1},x_{2n}] \in [F_{2n},F_{2n}],$$
then $w, \phi(w)\in \ker(\pi)$. Moreover, $\phi(w)$ cannot belong to any proper free factor of $F_{2n}$ because it is a test element. So for any $1 \leq i \leq 2n$, both $x_i$ and $x_i^{-1}$ occur at least once in the reduced presentation of $\phi(w)$ with respect to $\mathcal{S}$. Therefore, $|\phi(w)|_\mathcal{S} \geq 4n=|w|_\mathcal{S}$.
\end{proof}

\subsection{Whitehead graphs}

For any element $w\in F_r$, the \textit{(cyclic) Whitehead graph} $\mathrm{Wh}_{\mathcal{S}}(w)$ of $w$ with respect to $\mathcal{S}$, is a graph with the vertex set $\mathcal{S} \cup \mathcal{S}^{-1}$. Edges are added as follows: for any two letters $x,y \in \mathcal{S} \cup \mathcal{S}^{-1}$ with $x \neq y$, there is an edge joining $x$ and $y$, if the element $xy^{-1}$ occurs cyclically in the cyclic reduction $w_0$ of $w$ as a subword, i.e., $xy^{-1}$ is a subword of $w_0$, or $x$ is the last letter of $w_0$ and $y^{-1}$ is the first letter of $w_0$.  

By the definition of the Whitehead graph, we have:

\begin{lem}\label{Basic property of Wh graph} Let $F_r$ be a free group of rank $r$ with a basis $\mathcal{S}$.
\begin{enumerate}
\item If $w'$ is a conjugation of $w$ or $w^{-1}$ in $F_r$, then $$\mathrm{Wh}_\mathcal{S}(w')=\mathrm{Wh}_\mathcal{S}(w)=\mathrm{Wh}_\mathcal{S}(w^{-1});$$
\item  Let $u$ and $v$ be two nontrivial cyclically reduced words in $F_r$. If all of the two adjacent letter pairs in a cyclic permutation of $u$, appear as subwords in a cyclic permutation of $v$ or $v^{-1}$, then $\mathrm{Wh}_\mathcal{S}(u)$ is a subgraph of $\mathrm{Wh}_\mathcal{S}(v)$.
\end{enumerate}
\end{lem}

\subsection{Cut vertices}

A vertex $v$ of a graph is called a \textit{cut vertex} if the full subgraph spanned by vertices not equal to $v$ is disconnected. Below is the famous Whitehead's Cut-Vertex Lemma (see \cite{BBW24} for more information).

\begin{lem}[Whitehead, \cite{Whi36}]\label{Whitehead}
	Let $r\geq 2$. If $w$ is a simple element of $F_r$, then for any basis $\mathcal{S}$ of $F_r$, the cyclic Whitehead graph $\mathrm{Wh}_{\mathcal{S}}(w)$ contains a cut vertex.
\end{lem}

The following lemma shows the relationship between the cut vertices of the Whitehead graphs and their subgraphs.

\begin{lem}\label{cut vertex}
    Let $\mathrm{Wh}_{\mathcal{S}}(w_i) (i=1,2)$ be Whitehead graphs of two elements in a free group $F_r$ with respect to the same basis $\mathcal{S}$. If $\mathrm{Wh}_{\mathcal{S}}(w_1)$ is a subgraph of $\mathrm{Wh}_{\mathcal{S}}(w_2)$, then the cut vertices of $\mathrm{Wh}_{\mathcal{S}}(w_2)$ are also cut vertices of $\mathrm{Wh}_{\mathcal{S}}(w_1)$. In particular, $\mathrm{Wh}_{\mathcal{S}}(w_1)$ lacks cut vertices implies that $\mathrm{Wh}_{\mathcal{S}}(w_2)$ also lacks cut vertices.
\end{lem}

\begin{proof}
    According to the definition of the Whitehead graph, the vertex sets of $\mathrm{Wh}_{\mathcal{S}}(w_1)$ and $\mathrm{Wh}_{\mathcal{S}}(w_2)$ are both $\mathcal{S} \cup \mathcal{S}^{-1}$, the edge set of $\mathrm{Wh}_{\mathcal{S}}(w_1)$ is a subset of the edge set of $\mathrm{Wh}_{\mathcal{S}}(w_2)$. If $v$ is a cut vertex of $\mathrm{Wh}_{\mathcal{S}}(w_2)$, then the full subgraph $W''$ of $\mathrm{Wh}_{\mathcal{S}}(w_2)$ spanned by vertices not equal to $v$ is disconnected, and hence the corresponding full subgraph $W'$ of $\mathrm{Wh}_{\mathcal{S}}(w_1)$ is not connected again (because $W'$ is a subgraph of $W''$ with the same vertex set). So $v$ is also a cut vertex of $\mathrm{Wh}_{\mathcal{S}}(w_1)$.
\end{proof}

Moreover, Bestvina-Bridson-Wade showed the following.

\begin{lem}\cite[Proposition 2.2]{BBW24}\label{Bestvina0}
	Let $r\geq 2$. If $w \in F_r$ is not contained in a proper free factor and $|w|_\mathcal{S} \leq |\phi(w)|_\mathcal{S}$ for all $\phi \in \aut(F_r)$, then $\mathrm{Wh}_{\mathcal{S}}(w)$ contains no cut vertex.
\end{lem}

Note that the condition  $|u|_\mathcal{S} \leq |\phi(u)|_\mathcal{S}$ for every automorphism $\phi \in \aut(F_r)$ is equivalent to that $\mathcal{S}$ minimizes $|u|_\mathcal{S}$. Therefore, Lemma \ref{Bestvina0} can be restated as follows.

\begin{lem}\label{Bestvina}
	Let $r\geq 2$. If $w \in F_r$ is not simple and $\mathcal{S}$ minimizes $|w|_\mathcal{S}$, then $\mathrm{Wh}_{\mathcal{S}}(w)$ contains no cut vertex.
\end{lem}

\section{Main results}

We now study which elements are test elements specifically for monomorphisms of $F_r$, in other words, if such an element is fixed by $\phi \in \mon(F_r)$, then $\phi \in \aut(F_r)$.
Recall that Turner described such elements in \cite{Tu96}.

\begin{prop}[Turner, \cite{Tu96}]\label{Turner2} Let $F_r$ be a free group of rank $r$. Then
 \begin{enumerate}
     \item The test elements in $F_r$ are words not contained in proper retracts; 
     \item The test elements for monomorphisms in $F_r$ are the non-simple elements.
 \end{enumerate}
\end{prop}

Proposition \ref{Turner2} showed that a test element for monomorphisms can not be contained in a proper free factor, in fact, whether a given element lies in a proper free factor can be decided algorithmically by using some group-theoretic methods \cite{KM02, MVW07}. 

In this paper, by considering the natural action of the free group $F_r$ on its Cayley graph $\mathrm{Cay}(F_r,\mathcal{S})$,  we will give a sufficient condition in geometry for an element to be a test element for monomorphisms. 

\subsection{A key observation of Whitehead graphs}
For any nontrivial element $a\in F_r$, there exists a unique axis $X_a$ of the action of $a$ on the Cayley graph $\mathrm{Cay}(F_r,\mathcal{S})$. Note that $a$ acts on $X_a$ as a translation with length $|a_1|_\mathcal{S}$ where $a_1$ is the cyclically reduction of $a$. (See Figure \ref{fig:axis} for examples and \cite{CM87} for more information of group actions on $\Real$-trees).  The following proposition presents a key property of  Whitehead graphs of cyclically reduced elements.

\tikzset{every picture/.style={line width=0.75pt}} %set default line width to 0.75pt        
\begin{figure}
\begin{tikzpicture}[x=0.75pt,y=0.75pt,yscale=-1,xscale=1]
%uncomment if require: \path (0,457); %set diagram left start at 0, and has height of 457

%Straight Lines [id:da061675483834660705] 
\draw [line width=1.5]    (100.08,109.88) -- (276.08,110.66) ;
\draw [shift={(280.08,110.68)}, rotate = 180.25] [fill={rgb, 255:red, 0; green, 0; blue, 0 }  ][line width=0.08]  [draw opacity=0] (15.6,-3.9) -- (0,0) -- (15.6,3.9) -- cycle    ;
%Straight Lines [id:da6618147191053839] 
\draw    (190.08,29.48) -- (190.08,190.68) ;
%Straight Lines [id:da7359416150734883] 
\draw    (170.48,60.28) -- (209.68,60.68) ;
%Straight Lines [id:da7434261731973483] 
\draw    (171.29,160.52) -- (210.5,160.52) ;
%Straight Lines [id:da2965610370981008] 
\draw    (140.48,89.88) -- (140.08,130.28) ;
%Straight Lines [id:da3545639810856944] 
\draw    (240.08,89.88) -- (239.68,130.28) ;
%Straight Lines [id:da10566037588138444] 
\draw [line width=0.75]    (330.08,110.01) -- (510.08,110.81) ;
%Straight Lines [id:da6127379762875902] 
\draw    (419.25,29.82) -- (420.91,191.01) ;
%Straight Lines [id:da04829572633418189] 
\draw    (399.98,59.28) -- (439.18,59.68) ;
%Straight Lines [id:da615951400066151] 
\draw    (401,160.81) -- (440.2,160.81) ;
%Straight Lines [id:da6049808958915511] 
\draw    (370.48,90.88) -- (370.08,131.28) ;
%Straight Lines [id:da3333488896167297] 
\draw [line width=1.5]    (470.01,44.35) -- (470.5,178.85) ;
\draw [shift={(470,40.35)}, rotate = 89.79] [fill={rgb, 255:red, 0; green, 0; blue, 0 }  ][line width=0.08]  [draw opacity=0] (15.6,-3.9) -- (0,0) -- (15.6,3.9) -- cycle    ;
%Straight Lines [id:da8943385106101384] 
\draw    (451.5,140.81) -- (490.7,140.81) ;
%Straight Lines [id:da12746698899288633] 
\draw    (450,79.81) -- (489.2,79.81) ;

% Text Node
\draw (270.77,120.62) node [anchor=north west][inner sep=0.75pt]  [font=\normalsize,rotate=-359.41] [align=left] {$X_{s_1}$};
% Text Node
\draw (194.43,116.69) node [anchor=north west][inner sep=0.75pt]  [font=\small,rotate=-359.72] [align=left] {1};
% Text Node
\draw (145.63,115.89) node [anchor=north west][inner sep=0.75pt]  [font=\small,rotate=-359.72] [align=left] {$s_1^{-1}$};
% Text Node
\draw (244.43,118.89) node [anchor=north west][inner sep=0.75pt]  [font=\small,rotate=-359.72] [align=left] {$s_1$};
% Text Node
\draw (477.98,30.88) node [anchor=north west][inner sep=0.75pt]  [font=\normalsize,rotate=-359.41] [align=left] {$X_{s_1s_2s_1^{-1}}$};
% Text Node
\draw (424.5,116.98) node [anchor=north west][inner sep=0.75pt]  [font=\small,rotate=-359.13] [align=left] {1};
% Text Node
\draw (475.13,85.89) node [anchor=north west][inner sep=0.75pt]  [font=\small,rotate=-359.72] [align=left] {$s_1s_2$};
% Text Node
\draw (474.49,115.67) node [anchor=north west][inner sep=0.75pt]  [font=\small,rotate=-359.13] [align=left] {$s_1$};
% Text Node
\draw (475.63,148.39) node [anchor=north west][inner sep=0.75pt]  [font=\small,rotate=-359.72] [align=left] {$s_1s_2^{-1}$};
\end{tikzpicture}
\caption{Two examples of axes in the Cayley graph $\operatorname{Cay}(F(s_1,s_2),\{s_1,s_2\})$. The left figure shows the axis $X_{s_1}$ of $s_1$. The right figure shows the axis $X_{s_1s_2s_1^{-1}}=s_1X_{s_2}$ of the conjugate $s_1s_2s_1^{-1}$.} \label{fig:axis}
\end{figure}

\begin{prop}\label{non-cyc-reduced}
Let $w \in F_r$ be cyclically reduced. For any nontrivial element $a \in F_r$, if $|X_a \cap X_w| \geq |w|_\mathcal{S}+1$, then $\mathrm{Wh}_\mathcal{S}(w)$ is a subgraph of $\mathrm{Wh}_\mathcal{S}(a^k)$ for some $k \in \mathbb{Z}$. 
\end{prop}

\begin{proof}
If $|w|_\mathcal{S}\leq 1$, i.e.,  $w=1$ or $w\in \mathcal{S}\cup\mathcal{S}^{-1}$, then $\mathrm{Wh}_\mathcal{S}(w)$ has no edges and hence $\mathrm{Wh}_\mathcal{S}(w)$ is a subgraph of $\mathrm{Wh}_\mathcal{S}(a)$ clearly. Now we suppose $|w|_\mathcal{S}\geq 2$ and denote $$w \equiv s_1\cdots s_2,$$
 where $s_1, s_2 \in \mathcal{S}\cup\mathcal{S}^{-1} $ ($s_1\neq s_2^{-1}$) are the first and last letters of $w$ respectively.

\textbf{Case (1)}.  The nontrivial element $a \in F_r$ is not cyclically reduced. (See Figure \ref{fig: non-cyc-reduced}).

In this case, we can denote 
\begin{equation}\label{equ. a presentation}
a \equiv ta_1t^{-1},
\end{equation}
(equality as reduced words) where  $a_1$ is cyclically reduced and $t, a_1\neq 1$. Then the axis $X_a=t(X_{a_1})$ and hence the vertices of $X_a$ contain $$\{\cdots,ta_1^{-2},ta_1^{-1},t,ta_1,ta_1^{2},\cdots\}.$$
The translation length of $a$ on $X_a$ is $|a_1|_\mathcal{S}$. Note that $1 \notin X_a$ but $1 \in X_w$ because $w$ is cyclically reduced and $a$ is not.  So we have $1 \notin X_a \cap X_w$. Moreover, we claim that $$d(1, X_a \cap X_w)=|t|_\mathcal{S}.$$
Indeed, if $d(1, X_a \cap X_w)<|t|_\mathcal{S}$, then 
$$d(1, X_a \cap X_w)=d(1, ta_1^j)=|ta_1^j|_\mathcal{S}<|t|_\mathcal{S}$$
for some $j \in \Z \setminus \{0\}$, which contradicts with the presentation Equation  (\ref{equ. a presentation}) of $a$. (Notice that $ta_1^j$ and $t$ are both reduced).
	Hence the terminal point of the geodesic from 1 to $X_a \cap X_w$ is $t$.

Since $|X_a \cap X_w| \geq |w|_\mathcal{S}+1$ and $X_a \cap X_w$ is connected (not necessarily compact), there exists $k \in \Z$ such that $|a_1^k|_\mathcal{S} \geq |w|_\mathcal{S}+1$ and all of the adjacent pairs of letters in $w$, including $s_2s_1$, appear as subwords of $a_1^k$. (Note that the signum of $k$ is determined by the direction of the translation of the action of $a$ on $X_a$.) So by Lemma \ref{Basic property of Wh graph}, the Whitehead graph $\mathrm{Wh}_\mathcal{S}(w)$ is a subgraph of $\mathrm{Wh}_\mathcal{S}(a^k)$.\\

\begin{figure}[h]
\begin{center}
\setlength{\unitlength}{1mm}
\begin{picture}(115,40)(-33,-20)
\put(0,0){\line(1,0){48}}
\put(0,0){\line(-1, 1){10}}
\put(0,0){\line(-1, 0){5}}
\put(-10,0){\line(-1, 0){5}}
\put(48,0){\line(1, 0){15}}
\put(48,0){\line(1, -1){10}}
\put(-14.5,11){\makebox(0,0)[cc]{\begin{rotate}{-5}$\ddots$\end{rotate}}}
\put(58.2,-13.5){\makebox(0,0)[cc]{\begin{rotate}{-5}$\ddots$\end{rotate}}}
\put(66,0){\makebox(0,0)[cc]{$\cdots$}}
\put(-7.3,0){\makebox(0,0)[cc]{$\cdots$}}
\put(-13,0){\makebox(0,0)[cc]{$\bullet$}}
\put(-3,0){\makebox(0,0)[cc]{$\bullet$}}
\put(10,0){\makebox(0,0)[cc]{$\bullet$}}
\put(24,0){\makebox(0,0)[cc]{$\bullet$}}
\put(38,0){\makebox(0,0)[cc]{$\bullet$}}
\put(51,0){\makebox(0,0)[cc]{$\bullet$}}
\put(0,0){\makebox(0,0)[cc]{$\bigcdot$}}
\put(48,0){\makebox(0,0)[cc]{$\bigcdot$}}
\put(51,0){\vector(1,0){5}}
\put(51,-3){\vector(1,-1){5}}
\put(-10,10){\vector(1,-1){5}}
\put(1,0){\vector(1,0){5}}
\put(24,0){\vector(1,0){7}}
\put(-13,-2){\makebox(0,0)[ct]{1}}
\put(0,-2.5){\makebox(0,0)[ct]{$t$}}
\put(10,-1.5){\makebox(0,0)[ct]{$w^i$}}
\put(24,-1.5){\makebox(0,0)[ct]{$w^{i+1}$}}
\put(38,-1.5){\makebox(0,0)[ct]{$w^{i+2}$}}
\put(63,5){\makebox(0,0)[ct]{$X_w$}}
\put(-15,5){\makebox(0,0)[ct]{$X_w$}}
\put(63,-8){\makebox(0,0)[ct]{$X_a$}}
\put(-15,12){\makebox(0,0)[ct]{$X_a$}}
\put(24,-7){\makebox(0,0)[cc]{$\underbrace{\hspace{4.8cm}}$}}
\put(24,-12){\makebox(0,0)[cc]{$X_a \cap X_w$}}
\put(17,3){\makebox(0,0)[cc]{$\overbrace{\hspace{1.4cm}}$}}
\put(17,7){\makebox(0,0)[cc]{$|w|_{\mathcal{S}}$}}
\end{picture}
\end{center}
\caption{$a$ is not cyclically reduced}\label{fig: non-cyc-reduced}
\end{figure}
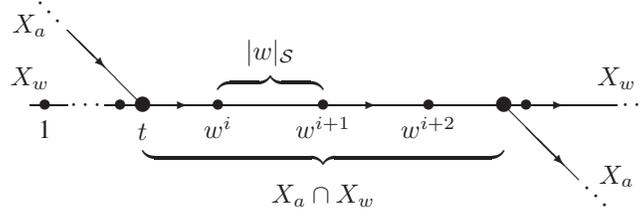

\textbf{Case (2)}.  The nontrivial element $a \in F_r$ is cyclically reduced  (see Figure \ref{fig: cyc-reduced-1} and  Figure \ref{fig: cyc-reduced-2}). Then $$1 \in X_a \cap X_w.$$ Let  $x,y \in F_r$ be the terminal vertices in the segment $X_a \cap X_w$ respectively, where $y$ is the vertex such that the geodesic from 1 to $y$ has the same direction as the geodesic from 1 to $w$. Then $$|x|_\mathcal{S}+|y|_\mathcal{S}=|X_a \cap X_w|\geq |w|_\mathcal{S}+1.$$

If $|y|_\mathcal{S} \geq |w|_\mathcal{S}+1$ or $|x|_\mathcal{S} \geq |w|_\mathcal{S}+1$, then for sufficiently large $k>0$, 
$$a^k\equiv ws_1\cdots, ~\quad or \quad a^k\equiv \cdots s_2w,$$ i.e., $ws_1$ is a prefix of $a^k$ or $w^{-1}s_2^{-1}$ is a prefix of $a^{-k}$. Note that in either case,  all of the adjacent letter pairs in $w$, including $s_2s_1$, appear as subwords of $a^k$. So $\mathrm{Wh}_\mathcal{S}(w)$ is a subgraph of $\mathrm{Wh}_\mathcal{S}(a^k)$  by Lemma \ref{Basic property of Wh graph}.

\begin{figure}[h]
\begin{center}
\setlength{\unitlength}{1mm}
\begin{picture}(115,40)(-33,-20)
\put(0,0){\line(1,0){48}}
\put(0,0){\line(-1, 1){10}}
\put(0,0){\line(-1, 0){10}}
\put(48,0){\line(1, 0){15}}
\put(48,0){\line(1, -1){10}}
\put(-14.5,11){\makebox(0,0)[cc]{\begin{rotate}{-5}$\ddots$\end{rotate}}}
\put(58.2,-13.5){\makebox(0,0)[cc]{\begin{rotate}{-5}$\ddots$\end{rotate}}}
\put(66,0){\makebox(0,0)[cc]{$\cdots$}}
\put(-12.5,0){\makebox(0,0)[cc]{$\cdots$}}
\put(-3,0){\makebox(0,0)[cc]{$\bullet$}}
\put(10,0){\makebox(0,0)[cc]{$\bullet$}}
\put(24,0){\makebox(0,0)[cc]{$\bullet$}}
\put(38,0){\makebox(0,0)[cc]{$\bullet$}}
\put(51,0){\makebox(0,0)[cc]{$\bullet$}}
\put(0,0){\makebox(0,0)[cc]{$\bigcdot$}}
\put(48,0){\makebox(0,0)[cc]{$\bigcdot$}}
\put(51,0){\vector(1,0){5}}
\put(51,-3){\vector(1,-1){5}}
\put(-10,10){\vector(1,-1){5}}
\put(1,0){\vector(1,0){5}}
\put(24,0){\vector(1,0){7}}
\put(10,-2.5){\makebox(0,0)[ct]{1}}
\put(24,-3){\makebox(0,0)[ct]{$w$}}
\put(38,-1.5){\makebox(0,0)[ct]{$w^2$}}
\put(0,-3){\makebox(0,0)[ct]{$x$}}
\put(48,-3){\makebox(0,0)[ct]{$y$}}
\put(63,5){\makebox(0,0)[ct]{$X_w$}}
\put(-15,5){\makebox(0,0)[ct]{$X_w$}}
\put(63,-8){\makebox(0,0)[ct]{$X_a$}}
\put(-15,12){\makebox(0,0)[ct]{$X_a$}}
\put(24,-7){\makebox(0,0)[cc]{$\underbrace{\hspace{4.8cm}}$}}
\put(24,-12){\makebox(0,0)[cc]{$X_a \cap X_w$}}
\put(17,3){\makebox(0,0)[cc]{$\overbrace{\hspace{1.4cm}}$}}
\put(17,7){\makebox(0,0)[cc]{$|w|_{\mathcal{S}}$}}
\end{picture}
\end{center}
\caption{$a$ is cyclically reduced with $|y|_\mathcal{S}\geq |w|_\mathcal{S}+1$}\label{fig: cyc-reduced-1}
\end{figure}
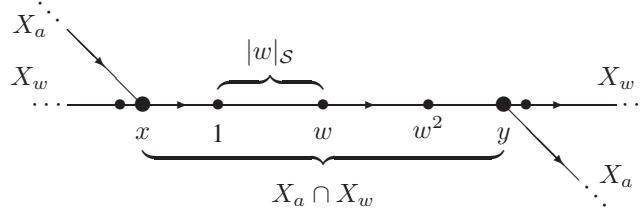

If $|y|_\mathcal{S} \leq |w|_\mathcal{S}$ and $|x|_\mathcal{S} \leq |w|_\mathcal{S}$, then the geodesic from 1 to $y$ is contained in the geodesic from 1 to $w$, so $y$ is a prefix of $w$. Let $w \equiv yb$ and $s \in \mathcal{S}\cup\mathcal{S}^{-1}$ the last letter of $y$. Since $$|x|_\mathcal{S}+|y|_\mathcal{S}=|X_a \cap X_w| \geq |w|_\mathcal{S}+1=|y|_\mathcal{S}+|b|_\mathcal{S}+1,$$ we have $|x|_\mathcal{S}\geq |b|_\mathcal{S}+1$. Moreover, note that $x$ is a prefix of $w^{-1}$,  so we can denote $x \equiv b^{-1}s^{-1}\cdots$. Then for some integer $k$ with $|k|$ large enough, $y$ is a prefix of $a^k$  and $x$ is a prefix of $a^{-k}$. Therefore
 $$a^k\equiv y\cdots x^{-1}\equiv  y\cdots sb$$ is cyclically reduced. Recall that $w \equiv yb$ and $s$ is the last letter of $y$, so $\mathrm{Wh}_\mathcal{S}(w)$ is a subgraph of $\mathrm{Wh}_\mathcal{S}(a^k)$  by Lemma \ref{Basic property of Wh graph}.
\end{proof}

\begin{figure}[h]
\begin{center}
\setlength{\unitlength}{1mm}
\begin{picture}(115,40)(-33,-20)
\put(0,0){\line(1,0){48}}
\put(14,0){\line(-1, 1){15}}
\put(0,0){\line(-1, 0){10}}
\put(48,0){\line(1, 0){15}}
\put(34,0){\line(1, -1){15}}
\put(-5.5,16){\makebox(0,0)[cc]{\begin{rotate}{-5}$\ddots$\end{rotate}}}
\put(49.2,-18.5){\makebox(0,0)[cc]{\begin{rotate}{-5}$\ddots$\end{rotate}}}
\put(66,0){\makebox(0,0)[cc]{$\cdots$}}
\put(-12.5,0){\makebox(0,0)[cc]{$\cdots$}}
\put(-8,0){\makebox(0,0)[cc]{$\bullet$}}
\put(8,0){\makebox(0,0)[cc]{$\bullet$}}
\put(24,0){\makebox(0,0)[cc]{$\bullet$}}
\put(40,0){\makebox(0,0)[cc]{$\bullet$}}
\put(56,0){\makebox(0,0)[cc]{$\bullet$}}
\put(14,0){\makebox(0,0)[cc]{$\bigcdot$}}
\put(34,0){\makebox(0,0)[cc]{$\bigcdot$}}
\put(45,0){\vector(1,0){5}}
\put(37,-3){\vector(1,-1){5}}
\put(4,10){\vector(1,-1){5}}
\put(-3,0){\vector(1,0){5}}
\put(13,0){\vector(1,0){7}}
\put(24,0){\vector(1,0){7}}
\put(-8,-1.5){\makebox(0,0)[ct]{$w^{-2}$}}
\put(8,-1.5){\makebox(0,0)[ct]{$w^{-1}$}}
\put(24,-2.5){\makebox(0,0)[ct]{1}}
\put(40,-3){\makebox(0,0)[ct]{$w$}}
\put(56,-1.5){\makebox(0,0)[ct]{$w^{2}$}}
\put(14,-3){\makebox(0,0)[ct]{$x$}}
\put(34,-3){\makebox(0,0)[ct]{$y$}}
\put(63,5){\makebox(0,0)[ct]{$X_w$}}
\put(-15,5){\makebox(0,0)[ct]{$X_w$}}
\put(49,-8){\makebox(0,0)[ct]{$X_a$}}
\put(-1,12){\makebox(0,0)[ct]{$X_a$}}
\put(24,-7){\makebox(0,0)[cc]{$\underbrace{\hspace{2cm}}$}}
\put(24,-12){\makebox(0,0)[cc]{$X_a \cap X_w$}}
\put(32,3){\makebox(0,0)[cc]{$\overbrace{\hspace{1.6cm}}$}}
\put(32,7){\makebox(0,0)[cc]{$|w|_{\mathcal{S}}$}}
\end{picture}
\end{center}
\caption{$a$ is cyclically reduced with $|x|_\mathcal{S}\leq |w|_\mathcal{S}$ and $|y|_\mathcal{S}\leq |w|_\mathcal{S}$}\label{fig: cyc-reduced-2}
\end{figure}

\subsection{Test elements for monomorphisms}\label{sect 3}

\begin{thm}\label{new simple element}
Let $w \in F_r (r\geq 2)$ be a non-simple element and let $\mathcal{S}$ be a basis of $F_r$ which minimizes $|w|_\mathcal{S}$. For any nontrivial $a \in F_r$, if $|X_a \cap X_w| \geq |w|_\mathcal{S}+1$, then $a$ is again non-simple.
\end{thm}

\begin{proof}
Since $|w|_\mathcal{S}$ is  minimal,  $w$ is cyclically reduced by Lemma \ref{min length lem}. Then by Proposition \ref{non-cyc-reduced}, there is an integer $k \in \mathbb{Z}$ such that $\mathrm{Wh}_\mathcal{S}(w)$ is a subgraph of $\mathrm{Wh}_\mathcal{S}(a^k)$. 
Moreover, since $w \in F_r$ is non-simple, the Whitehead graph $\mathrm{Wh}_\mathcal{S}(w)$ lacks cut vertices by Lemma \ref{Bestvina}. So $\mathrm{Wh}_\mathcal{S}(a^k)$ also lacks cut vertices by Lemma \ref{cut vertex}. Then Lemma \ref{Whitehead} implies that $a^k$ is again non-simple. Therefore,  $a$ is non-simple.
\end{proof}

Note that test elements are always test elements for monomorphisms, and the test elements for monomorphisms in $F_r$ are exactly the non-simple elements (see Proposition \ref{Turner2}), so Theorem \ref{new simple element} is equivalent to the following.

\begin{thm}\label{test for Mon}
Let $w \in F_r(r\geq 2)$ be a test element (for monomorphisms) and let $\mathcal{S}$ be a basis of $F_r$ which minimizes $|w|_\mathcal{S}$. For any nontrivial $a \in F_r$, if $|X_a \cap X_w| \geq |w|_\mathcal{S}+1$, then $a$ is again a test element for monomorphisms. 
\end{thm}

By Example \ref{commutator is test element} and Proposition \ref{commutator is minimal}, the elements $u=x_1^2 x_2^2\cdots x_r^2\in F_r$ 
and $w=[x_1,x_2][x_3,x_4]\cdots[x_{2n-1},x_{2n}]\in F_{2n}$ are test elements (and hence test elements for monomorphisms), and have minimal length. Therefore, as corollaries, we have:

\begin{cor}\label{u1}
Let $F_{2n}$ be a free group of rank $2n$ with basis $\mathcal{S}=\{x_1, \ldots,x_{2n}\}$. Then any cyclically reduced word $a\in F_{2n}$ containing the subword
    $$[x_1,x_2][x_3,x_4]\cdots[x_{2n-1},x_{2n}]x_1$$ is non-simple, or equivalently,  is a test element for monomorphisms.
\end{cor}
\begin{proof} Let $w=[x_1,x_2][x_3,x_4]\cdots[x_{2n-1},x_{2n}]$. By the assumption, we can denote the cyclically reduced word $a\equiv a'wx_1a''$ ($a', a''$ may be trivial). Then  $wx_1a''a'$ is also cyclically reduced and the intersection of the two axes satisfies
$$|X_{wx_1a''a'} \cap X_w| \geq |w|_\mathcal{S}+1.$$
So $wx_1a''a'$ is non-simple by Theorem \ref{test for Mon}. Note that a conjugate of a simple element is again simple,  so $a=a'\cdot wx_1a''a'\cdot a'^{-1}$ is non-simple.
\end{proof}

By a parallel argument as above, we have the following.

\begin{cor}\label{stronger version of Kapovich}
Let $F_n$ be a free group of rank $n$ with basis $\mathcal{S}=\{x_1,\ldots, x_{n}\}$. Then any cyclically reduced word in $F_n$ containing the subword $x_1^2 x_2^2 \cdots x_n^2 x_1$ is non-simple, or equivalently, is a test element for monomorphisms.
\end{cor}

\begin{rem} 
Both the above corollaries can be proved directly from the Whitehead-graph arguments without using Theorem \ref{test for Mon}, we present them here to demonstrate the potential applications of this theorem. Moreover, Corollary \ref{stronger version of Kapovich} is stronger than Gupta-Kapovich \cite[Corollary 2.18]{GK19}:
\textit{Let $F_n$ be a free group of rank $n$ with basis $\mathcal{S}=\{x_1,\ldots, x_{n}\}$. If a cyclically reduced word in $F_n$ contains the subword $x_1^2 x_2^2 \cdots x_n^2 x_1^2$, then it is non-simple (and hence not primitive) in $F_n$.}
\end{rem}
Finally, we have:

\begin{exa}
Suppose $F_{r}$ is a free group of rank $r$ with basis $\mathcal{S}=\{x_1,x_2,\ldots,x_{r}\}$.  Then 
$$u_1=x_1^2 x_2^2\cdots x_r^2 x_1 ~ \quad \mathrm{and}~\quad u_2=[x_1,x_2][x_3,x_4]\cdots[x_{2n-1}, x_{2n}]x_1 ~\mathrm{if} ~r=2n,$$
are test elements for monomorphisms but not test elements.

\end{exa}
\begin{proof}
By Corollary \ref{u1} and Corollary \ref{stronger version of Kapovich}, both of $u_1$ and $u_2$ are test elements for monomorphisms. To show neither of them are test elements, let $\phi_i: F_{r} \to \langle u_i \rangle \leq F_r$ ($i=1,2$) defined as
$$\phi_1(x_1)=u_1^{-1}, ~\phi_1(x_2)=u_1^2,~~\phi_1(x_j)=1, ~j\geq 3;$$
$$\phi_2(x_1)=\phi_2(x_2)=\ldots =\phi_2(x_{2n})=u_2.$$
Then $\phi_i(u_i)=u_i$. Note that both of $\phi_1$ and $\phi_2$ are endomorphisms but not automorphisms. Therefore, neither of $u_1$ and $u_2$ are test elements.
\end{proof}
%%%%%%%%%%%%%%%%%%55================================================================================

\noindent\textbf{Acknowledgements.}  The authors would like to thank the editor and the referee for the careful reading and for the helpful comments and suggestions.

%References-----------------------------------------------------------------

\end{document}